\documentclass[12pt, reqno]{amsart}
\usepackage{amsmath, amsthm, amscd, amsfonts, amssymb, graphicx, color}
\usepackage[bookmarksnumbered, colorlinks, plainpages]{hyperref}

\input{mathrsfs.sty}

\newtheorem{theorem}{Theorem}[section]
\newtheorem{lemma}[theorem]{Lemma}

\newtheorem{corollary}[theorem]{Corollary}
\theoremstyle{definition}

\newtheorem{example}[theorem]{Example}

\theoremstyle{remark}
\newtheorem{remark}[theorem]{Remark}
\numberwithin{equation}{section}
\begin{document}

\title [Reverses of Ando's and  H\"{o}lder--Macarty's inequalities]{Reverses of Ando's and  H\"{o}lder--Macarty's inequalities}

\author[ M. Hajmohamadi, R. Lashkaripour,  M. Bakherad ]{ M. Hajmohamadi$^1$, R. Lashkaripour$^2$ and M. Bakherad$^3$}

\address{$^1$$^{,2}$$^{,3}$ Department of Mathematics, Faculty of Mathematics, University of Sistan and Baluchestan, Zahedan, I.R.Iran.}

\email{$^{1}$monire.hajmohamadi@yahoo.com}
\email{$^2$lashkari@hamoon.usb.ac.ir}
\email{$^{3}$mojtaba.bakherad@yahoo.com; bakherad@member.ams.org}


\begin{abstract}
In this paper, we give some reverse-types of Ando's and H\"{o}lder-McCarthy's inequalities for positive linear maps, and positive invertible operators. For our purpose, we use a recently improved Young inequality and its reverse.
\end{abstract} \maketitle
\section{Introduction and preliminaries}
\noindent Let $\mathcal{B}(\mathcal{H})$ be the $C^*$-algebra of all bounded linear operators on  a complex Hilbert space $\mathcal{H}$ with the operator norm $\|\cdot\|$ and the identity $I$ also $\mathcal{M}_{n}(\mathcal{C})$ denotes the space of all $n\times n$ complex matrices. For an operator $A \in
\mathcal{B}(\mathcal{H})$, we write $A\geq 0$ if $A$ is positive, and
$A>0$ if $A$ is positive invertible. For $A, B \in \mathcal{B}(\mathcal{H})$, we say $A\geq B$ if $A-B\geq0$. The Gelfand map $f(t)\mapsto f(A)$ is an
isometrical $*$-isomorphism between the $C^*$-algebra
$C({\rm sp}(A))$ of continuous functions on the spectrum ${\rm sp}(A)$
of a selfadjoint operator $A$ and the $C^*$-algebra generated by $A$ and $I$. A linear map $\Phi$ on $\mathcal{B}(\mathcal{H})$ is positive if $\Phi(A)\geq0$ whenever $A\geq0$. It is said to be unital if $\Phi(I)=I$. A continuous function $f:J\rightarrow \mathcal{R}$ is operator concave if
\begin{align*}
f(\alpha A+(1-\alpha)B) \geq \alpha f(A)+(1-\alpha)f(B)
\end{align*}
 for all selfadjoint $A, B\in \mathcal{B}(\mathcal{H})$ with spectra in $J$ and all $\alpha\in [0,1]$.\\
 The well-known Young's inequality says that for positive real numbers $a, b$ and $0\leq t \leq1$, we have $a^{t}b^{1-t}\leq t a+(1-t)b$. Refinements and reverses of this inequality  are proven in \cite{bakh, dr.lashkari, sa, sab1, sab} and  references therein. Also F. Kittaneh et al. in \cite{kit} obtained the following improvement of the Young inequality, for any positive definite matrices $A,B\in\mathcal{M}_{n}(\mathcal{C})$
 {\footnotesize\begin{align}\label{kitt}
 A^{1-t}B^{t}+r(A+B-2A\sharp B)\leq(1-t)A+t B\leq A^{1-t}B^{t}+R(A+B-2A\sharp B),
 \end{align}}
 where $t\in [0,1], r=\min\{t,1-t\}$ and $R=\max\{t,1-t\}$.\\
 Zhao et al. \cite{zh} obtained a refinement of Young's inequality and its reverse as follows:\\
 $({\rm i})$ for $0< t\leq\frac{1}{2}$,
 \begin{align}\label{ww}
 r_{0}(\sqrt[4]{ab}-\sqrt{a})^{2}+r(\sqrt{a}-\sqrt{b})^{2}&+a^{1-t}b^{t}\leq(1-t)a+tb\nonumber\\&\leq R(\sqrt{a}-\sqrt{b})^{2}-r_{0}(\sqrt[4]{ab}-\sqrt{b})^{2}+a^{1-t}b^{t}
 \end{align}
 $({\rm ii})$ for $\frac{1}{2}<t<1$,
 \begin{align*}
 r_{0}(\sqrt[4]{ab}-\sqrt{b})^{2}+R(\sqrt{a}-\sqrt{b})^{2}&+a^{1-t}b^{t}\leq(1-t)a+tb\nonumber\\&\leq r(\sqrt{a}-\sqrt{b})^{2}-r_{0}(\sqrt[4]{ab}-\sqrt{a})^{2}+a^{1-t}b^{t},
 \end{align*}
 where $r=\min\{t,1-t\}, R=\max\{t,1-t\}$ and $r_{0}=\min\{2r,1-2r\}$.\\
Sababheh et al.  \cite{sa, sab}   established  some refinements and reverses of  Young's inequality as follows:\\
$({\rm i})$  for $0\leq t\leq\frac{1}{2}$,
\begin{align}\label{6}
S_{N}(t;a,b)\leq t a+(1-t)b-a^t b^{1-t}\leq (1-t)(\sqrt{a}-\sqrt{b})^2-S_{N}(2t;\sqrt{ab},a)
\end{align}
$({\rm ii})$ for $\frac{1}{2}\leq t\leq1$,
\begin{align*}
S_{N}(t;a,b)\leq t a+(1-t)b-a^t b^{1-t}\leq t(\sqrt{a}-\sqrt{b})^2-S_{N}(2-2t;\sqrt{ab},b),
\end{align*}
where \\
{\footnotesize\begin{align*}
S_{N}(t;a,b)=\sum_{j=1}^{N}s_{j}(t)\Big(\sqrt[2^{j}]{b^{2^{j-1}-k_{j}}a^{k_{j}}}-
\sqrt[2^{j}]{a^{k_{j}+1}b^{2^{j-1}-k_{j}-1}}\Big)^2,
\end{align*}}
 $s_{j}(t)=\Big((-1)^{r_{j}}2^{j-1}t+(-1)^{r_{j}+1}[\frac{r_{j}+1}{2}]\Big)$, $r_j=[2^{j}t]$ and $k_{j}=[2^{j-1}t]$. Here $[x]$ is the greatest integer less than or equal to $x$.\\
Let $A, B\in\mathcal{B}(\mathcal{H})$ be positive. The operator t-weighted arithmetic, geometric, and harmonic means of operators $A, B$ are defined by $A\nabla_{t}B=(1-t)A+tB$, $A\sharp_{t}B=A^{\frac{1}{2}}(A^{-\frac{1}{2}}BA^{-\frac{1}{2}})^{t}A^{\frac{1}{2}}$ and $A!_{t}B=((1-t)A^{-1}+tB^{-1})^{-1}$ respectively. In particular, for $t =\frac{1}{2}$
we get the usual operator arithmetic mean $\nabla$, the geometric mean $\sharp$ and the harmonic mean $!$.\\
\section*{Results and discussion}
For positive real numbers $a_i$ and $b_i\,\,(i=1,2,\dots,n)$  the H\"{o}lder inequality states  that
\begin{align}\label{ho}
\sum_{i=1}^{n}a_{i}^{1/p}b_{i}^{1/q}\leq\left(\sum_{i=1}^{n}a_{i}\right)^{1/p}\left(\sum_{i=1}^{n}b_{i}\right)^{1/q}
\end{align}
for $p,q>1$ such that $\frac{1}{p}+\frac{1}{q}=1$. If $p=q=2$ in \eqref{ho}, then we get the Cauchy-Schwarz inequality.
The H\"{o}lder inequality for positive operators  $A_i$ and $B_i\,\,(i=1,2,\dots,n)$ is
\begin{align*}
\sum_{i=1}^{n}A_{i}\sharp_{t}B_{i}\leq\left(\sum_{i=1}^{n}A_{i}\right)\sharp_{t}\left(\sum_{i=1}^{n}B_{i}\right),
\end{align*}
where $0\leq t\leq1$. In the case $t=\frac{1}{2}$, we get the operator Cauchy-Schwarz inequality. For further information about  the H\"{o}lder and Cauchy-Schwarz inequalities   we refer the
reader to \cite{bakh2, Bha,  Fuj, mm, Lee, zh2} and references therein.
Ando \cite{And} proved that if $\Phi$ is a positive linear map, then for positive operators $A, B\in\mathcal{B}(\mathcal{H})$ and $t\in[0,1]$, we have
\begin{align*}
\Phi(A\sharp_t B)\leq\Phi(A)\sharp_t\Phi(B).
\end{align*}
 Recently, some authors expressed several reverse-types  Ando's inequality (see \cite{se,seo}).
 \\H\"{o}lder-McCarthy's inequality says  that for any positive operator $A$ and any unit vector $x\in \mathcal H$, we have
\begin{align}\label{mc1}
\langle A^{t}x,x\rangle\leq\langle Ax,x\rangle^{t}, \qquad\qquad 0\leq t\leq1.
\end{align}
Furuta \cite{fur} showed that this inequality is equivalent to Young's inequality.\\
\section*{Conclutions}
In this paper, we establish a  reverse of Ando's inequality for positive(non-unital)linear maps  and  positive definite matrices by using an inequality due to Sababheh. We obtain some reverses of matrix H\"{o}lder and Cauchy-Schwarz inequalities and a reverse of inequality \eqref{mc1} for $t\in (0,\frac{1}{2}]$ as following:
{\footnotesize\begin{align*}
\langle Tx,x&\rangle^{t}-\langle T^{t}x,x\rangle\\&
\leq2R\left(\langle Tx,x\rangle^\frac{1}{2}-\langle T^\frac{1}{2}x,x\rangle\right)-r_{0}\left(\langle T^\frac{1}{2}x,x\rangle+
\langle Tx,x\rangle^\frac{1}{2}-2\langle T^\frac{1}{4}x,x\rangle\langle Tx,x\rangle^\frac{1}{4}\right).
\end{align*}}
\section*{Methods}
We use  the properties of inner product and the inequalities obtained in \cite{sab} and \cite{zh}.
\section{Main results}
\bigskip To prove our first result, we need the following lemmas.
\begin{lemma}\cite{sab}\label{sab1k}
 Let $A, B\in\mathcal{M}_{n}(\mathcal{C})$ be positive definite matrices and $t\in [0,1]$. Then
\begin{align}\label{a2}
\sum_{j=1}^{N}s_{j}(t)\left(A\sharp_{\alpha_{j}(t)}B+A\sharp_{2^{1-j}+\alpha_{j}(t)}B-2A\sharp_{2^{-j}+\alpha_{j}(t)}B\right)+A\sharp_{t}B\leq A\nabla_{t}B,
\end{align}
where $\alpha_{j}(t)=\frac{k_{j}(t)}{2^{j-1}}$.
\end{lemma}
For $N=2$, we have the following lemma that is shown in \cite{zh} for positive invertible operators.
\begin{lemma}\label{001}
Let $A, B\in\mathcal{B}(\mathcal{H})$ be positive invertible operators and $t\in [0,1]$. Then\\
 $(\rm i)$ If $0<t\leq\frac{1}{2}$, then
 \begin{align}\label{a2}
r_{0}(A\sharp B-2A\sharp_{\frac{1}{4}}B+A)+2t(A\nabla B-A\sharp B)+A\sharp_{t}B\leq A\nabla_{t}B.
 \end{align}
 $(\rm ii)$ If $\frac{1}{2}<t<1$, then
 \begin{align}\label{a3}
 r_{0}(A\sharp B-2A\sharp_{\frac{3}{4}}B+B)+2(1-t)(A\nabla B-A\sharp B)+A\sharp_{t}B\leq A\nabla_{t}B,
 \end{align}
 where $r=\min\{\nu,1-\nu\}$ and $r_{0}=\min\{2r,1-2r\}$.
\end{lemma}
\begin{lemma}\cite{sab}\label{sab2k}
Let $A, B\in\mathcal{M}_{n}(\mathcal{C})$ be positive definite matrices and $t\in [0,1]$.\\
 $(\rm i)$ If $0\leq t\leq\frac{1}{2}$, then
 {\footnotesize\begin{align}\label{a4}
A\nabla_{t}B\leq &A\sharp_{t}B+2(1-t)(A\nabla B-A\sharp B)\nonumber\\&-\sum_{j=1}^{N}s_{j}(2t)\left(A\sharp_{1-\beta_{j}(t)}B+A\sharp_{1+2^{-j}-\beta_{j}(t)}B-2A\sharp_{1-2^{-j-1}-\beta_{j}(t)}B\right).
 \end{align}}
 $(\rm ii)$ If $\frac{1}{2}\leq t\leq1$, then
  {\footnotesize\begin{align*}
 A\nabla_{t}B&\leq A\sharp_{t}B+2t(A\nabla B-A\sharp B)\\&-\sum_{j=1}^{N}s_{j}(2-2t)\left(A\sharp_{\gamma_{j}(t)}B+A\sharp_{\gamma_{j}(t)+2^{1-j}}B-2A\sharp_{\gamma_{j}(t)+2^{-j}}B\right),
 \end{align*}}
where $\beta_{j}(t)=2^{-j}k_{j}(2t)$ and $\gamma_{j}(t)=2^{1-j}k_{j}(2-2t)$.
\end{lemma}
\begin{remark}
 By using functional calculus and numerical inequalities in \cite{kit, sab}, we can extend inequality \eqref{kitt}, Lemmas \ref{sab1k} and \ref{sab2k} for positive invertible operators.
\end{remark}
For $N=2$, we have the following lemma that is shown in \cite{zh} for positive invertible operators.
\begin{lemma}\label{002}
Let $A, B\in\mathcal{B}(\mathcal{H})$ be positive invertible operators and $t\in [0,1]$.\\
 $(\rm i)$ If $0<t\leq\frac{1}{2}$, then
 \begin{align}\label{a2}
 A\nabla_{t}B \leq A\sharp_{t}B+2(1-t)(A\nabla B-A\sharp B)-r_{0}(A\sharp B-2A\sharp_{\frac{3}{4}}B+B).
 \end{align}
 $(\rm ii)$ If $\frac{1}{2}<t<1$, then
 \begin{align}\label{a3}
 A\nabla_{t}B \leq A\sharp_{t}B+2t(A\nabla B-A\sharp B)-r_{0}(A\sharp B-2A\sharp_{\frac{1}{4}}B+A),
 \end{align}
 where $r=\min\{\nu,1-\nu\}$ and $r_{0}=\min\{2r,1-2r\}$.
\end{lemma}
Now, we obtain a reverse of Ando's inequality for positive invertible operators as follows.
\begin{theorem}\label{1}
Let $A, B\in\mathcal{B}(\mathcal{H})$ be positive invertible operators, $\Phi$ be a positive linear map and  $t\in [0,1]$.\\
\\
 $(\rm i)$ If $0\leq t\leq\frac{1}{2}$, then
 {\footnotesize\begin{align}\label{i10}
\Phi(A)&\sharp_{t}\Phi(B)-\Phi(A\sharp_{t}B)\leq2R(\Phi(A)\sharp\Phi(B)-\Phi(A\sharp B)+\frac{1}{2}(\Phi(A)+\Phi(B)-2\Phi(A)\sharp\Phi(B)))\nonumber\\&-\sum_{j=1}^{N}s_{j}(2t)\left(\Phi(A\sharp_{1-\beta_{j}(t)}B)+\Phi(A\sharp_{1+2^{-j}-\beta_{j}(t)}B)-
2\Phi(A\sharp_{1-2^{-j-1}-\beta_{j}(t)}B)\right)\nonumber\\&
-\sum_{j=1}^{N}s_{j}(t)\left(\Phi(A)\sharp_{\alpha_{j}(t)}\Phi(B)+
\Phi(A)\sharp_{2^{1-j}+\alpha_{j}(t)}\Phi(B)-2\Phi(A)\sharp_{2^{-j}+\alpha_{j}(t)}\Phi(B)\right).
\end{align}}
$(\rm ii)$ If $\frac{1}{2}\leq t\leq1$, then
 {\footnotesize\begin{align}\label{i20}
\Phi(A)&\sharp_{t}\Phi(B)-\Phi(A\sharp_{t}B)\leq2R(\Phi(A)\sharp\Phi(B)-\Phi(A\sharp B)+\frac{1}{2}(\Phi(A)+\Phi(B)-2\Phi(A)\sharp\Phi(B)))\nonumber\\&-\sum_{j=1}^{N}s_{j}(2-2t)\left(\Phi(A\sharp_{\gamma_{j}(t)}B)+
\Phi(A\sharp_{\gamma_{j}(t)+2^{1-j}}B)-2\Phi(A\sharp_{\gamma_{j}(t)+2^{-j}}B)\right)\nonumber\\&
-\sum_{j=1}^{N}s_{j}(t)\Big(\Phi(A)\sharp_{\alpha_{j}(t)}\Phi(B)+
\Phi(A)\sharp_{2^{1-j}+\alpha_{j}(t)}\Phi(B)-2\Phi(A)\sharp_{2^{-j}+\alpha_{j}(t)}\Phi(B)\Big),
\end{align}}
where $\alpha_{j}(t)=\frac{k_{j}(t)}{2^{j-1}}$, $\beta_{j}(t)=2^{-j}k_{j}(2t)$, $\gamma_{j}(t)=2^{1-j}k_{j}(2-2t)$ and  $R=\max\{t,1-t\}$.
\end{theorem}
\begin{proof}
The proof of inequality \eqref{i20} is similar to the proof of inequality \eqref{i10}. Thus we only  prove inequality \eqref{i10}.\\
Let $0\leq t\leq\frac{1}{2}$. Applying inequalities \eqref{a2} and \eqref{a4}, we have
{\footnotesize\begin{align}\label{p1}
&\sum_{j=1}^{N}s_{j}(t)(A\sharp_{\alpha_{j}(t)}B+A\sharp_{2^{1-j}+\alpha_{j}(t)}B-2A\sharp_{2^{-j}+\alpha_{j}(t)}B)\nonumber\\&\leq A\nabla_{t} B-A\sharp_{t}B\nonumber\\&\leq 2R(A\nabla B-A\sharp B)
-\sum_{j=1}^{N}s_{j}(2t)\left(A\sharp_{1-\beta_{j}(t)}B+A\sharp_{1+2^{-j}-\beta_{j}(t)}B-2A\sharp_{1-2^{-j-1}-\beta_{j}(t)}B\right).
\end{align}}
Now, using the positive linear map $\Phi$ on \eqref{p1}, we get
{\footnotesize\begin{align}\label{p2}
\sum_{j=1}^{N}&s_{j}(t)\left(\Phi(A\sharp_{\alpha_{j}(t)}B)+\Phi(A\sharp_{2^{1-j}+\alpha_{j}(t)}B)-2\Phi(A\sharp_{2^{-j}+\alpha_{j}(t)}B)\right)
+\Phi(A\sharp_{t}B)\nonumber\\&\leq \Phi(A)\nabla_{t} \Phi(B)\nonumber\\&\leq 2R\left(\Phi(A)\nabla\Phi(B)-\Phi(A\sharp B)\right)+\Phi(A\sharp_{t}B)\nonumber\\&-\sum_{j=1}^{N}s_{j}(2t)\left(\Phi(A\sharp_{1-\beta_{j}(t)}B)+
\Phi(A\sharp_{1+2^{-j}-\beta_{j}(t)}B)-2\Phi(A\sharp_{1-2^{-j-1}-\beta_{j}(t)}B)\right).
\end{align}}
Moreover, if we replace $A$ and $B$ by $\Phi(A)$ and $\Phi(B)$ in inequality \eqref{p1}, respectively, then
{\footnotesize\begin{align}\label{p3}
&\sum_{j=1}^{N}s_{j}(t)\left(\Phi(A)\sharp_{\alpha_{j}(t)}\Phi(B)+\Phi(A)\sharp_{2^{1-j}+
\alpha_{j}(t)}\Phi(B)-2\Phi(A)\sharp_{2^{-j}+\alpha_{j}(t)}\Phi(B)\right)\nonumber\\&
+\Phi(A)\sharp_{t}\Phi(B)\nonumber\\&
\leq\Phi(A)\nabla_{t}\Phi(B)\nonumber\\&
\leq2R\left(\Phi(A)\nabla\Phi(B)-\Phi(A\sharp B)\right)+\Phi(A)\sharp_{t}\Phi(B)\nonumber\\&-\sum_{j=1}^{N}s_{j}(2t)\left(\Phi(A)\sharp_{1-\beta_{j}(t)}\Phi(B)+
\Phi(A)\sharp_{1+2^{-j}-\beta_{j}(t)}\Phi(B)-2\Phi(A)\sharp_{1-2^{-j-1}-\beta_{j}(t)}\Phi(B)\right).
\end{align}}
From the first inequality of \eqref{p3} and the second inequality of \eqref{p2}, we have
{\footnotesize\begin{align*}
\sum_{j=1}^{N}&s_{j}(t)\left(\Phi(A)\sharp_{\alpha_{j}(t)}\Phi(B)+\Phi(A)\sharp_{2^{1-j}+
\alpha_{j}(t)}\Phi(B)-2\Phi(A)\sharp_{2^{-j}+\alpha_{j}(t)}\Phi(B)\right)\\&
+\Phi(A)\sharp_{t}\Phi(B)\\&
\leq\Phi(A)\nabla_{t}\Phi(B)\\&\leq 2R(\Phi(A)\nabla\Phi(B)-\Phi(A\sharp B))+\Phi(A\sharp_{t}B)\\&-\sum_{j=1}^{N}s_{j}(2t)\left(\Phi(A\sharp_{1-\beta_{j}(t)}B)+
\Phi(A\sharp_{1+2^{-j}-\beta_{j}(t)}B)-2\Phi(A\sharp_{1-2^{-j-1}-\beta_{j}(t)}B)\right),
\end{align*}}
which implies that
{\footnotesize\begin{align*}
\sum_{j=1}^{N}&s_{j}(t)\left(\Phi(A)\sharp_{\alpha_{j}(t)}\Phi(B)+\Phi(A)\sharp_{2^{1-j}+
\alpha_{j}(t)}\Phi(B)-2\Phi(A)\sharp_{2^{-j}+\alpha_{j}(t)}\Phi(B)\right)\\&
+\Phi(A)\sharp_{t}\Phi(B)\\&
\leq 2R\left(\Phi(A)\nabla\Phi(B)-\Phi(A\sharp B)\right)+\Phi(A\sharp_{t}B)\\&-\sum_{j=1}^{N}s_{j}(2t)\left(\Phi(A\sharp_{1-\beta_{j}(t)}B)+
\Phi(A\sharp_{1+2^{-j}-\beta_{j}(t)}B)-2\Phi(A\sharp_{1-2^{-j-1}-\beta_{j}(t)}B)\right).
\end{align*}}
Therefore with applying inequality \eqref{kitt}, we get
{\footnotesize\begin{align*}
\Phi(A)&\sharp_{t}\Phi(B)-\Phi(A\sharp_{t}B)\leq2R(\Phi(A)\nabla\Phi(B)-\Phi(A\sharp B))\\&-\sum_{j=1}^{N}s_{j}(2t)\left(\Phi(A\sharp_{1-\beta_{j}(t)}B)+\Phi(A\sharp_{1+2^{-j}-\beta_{j}(t)}B)-
2\Phi(A\sharp_{1-2^{-j-1}-\beta_{j}(t)}B)\right)\\&
-\sum_{j=1}^{N}s_{j}(t)\left(\Phi(A)\sharp_{\alpha_{j}(t)}\Phi(B)+
\Phi(A)\sharp_{2^{1-j}+\alpha_{j}(t)}\Phi(B)-2\Phi(A)\sharp_{2^{-j}+\alpha_{j}(t)}\Phi(B)\right)\\&
\leq2R(\Phi(A)\sharp\Phi(B)-\Phi(A\sharp B)+\frac{1}{2}(\Phi(A)+\Phi(B)-2\Phi(A)\sharp\Phi(B)))\\&-\sum_{j=1}^{N}s_{j}(2t)\left(\Phi(A\sharp_{1-\beta_{j}(t)}B)+\Phi(A\sharp_{1+2^{-j}-\beta_{j}(t)}B)-
2\Phi(A\sharp_{1-2^{-j-1}-\beta_{j}(t)}B)\right)\\&
-\sum_{j=1}^{N}s_{j}(t)\left(\Phi(A)\sharp_{\alpha_{j}(t)}\Phi(B)+
\Phi(A)\sharp_{2^{1-j}+\alpha_{j}(t)}\Phi(B)-2\Phi(A)\sharp_{2^{-j}+\alpha_{j}(t)}\Phi(B)\right).
\end{align*}}
\end{proof}
Similarly for $N=2$ by applying Lemma \ref{001} and Lemma \ref{002}, we can obtain a reverse of Ando's inequality for positive invertible operators.
\begin{corollary}
Let $A, B\in\mathcal{B}(\mathcal{H})$ be positive invertible operators, $\Phi$ be a positive linear map and  $t\in [0,1]$.\\
$(\rm i)$ If $0< t\leq\frac{1}{2}$, then
 {\footnotesize\begin{align}\label{i1}
\Phi(A)\sharp_{t}\Phi(B)-\Phi(A\sharp_{t}B)&\leq2R(\Phi(A)\sharp\Phi(B)-\Phi(A\sharp B)+\frac{1}{2}(\Phi(A)+\Phi(B)-2\Phi(A)\sharp\Phi(B)))\nonumber\\&
-r_{0}(\Phi(A\sharp B)+\Phi(B)-2\Phi(A\sharp_{\frac{3}{4}}B))\nonumber\\&
-r_{0}(\Phi(A)\sharp \Phi(B)+\Phi(A)-2(\Phi(A)\sharp_{\frac{1}{4}}\Phi(B)))\nonumber\\&
\leq2R(\Phi(A)\sharp\Phi(B)-\Phi(A\sharp B)+\frac{1}{2}(\Phi(A)+\Phi(B)-2\Phi(A)\sharp\Phi(B))),
\end{align}}
$(\rm ii)$ if $\frac{1}{2}< t< 1$, then
 {\footnotesize\begin{align}\label{i2}
\Phi(A)\sharp_{t}\Phi(B)-\Phi(A\sharp_{t}B)&\leq2R(\Phi(A)\sharp\Phi(B)-\Phi(A\sharp B)+\frac{1}{2}(\Phi(A)+\Phi(B)-2\Phi(A)\sharp\Phi(B)))\nonumber\\&
-r_{0}(\Phi(A)\sharp \Phi(B)+\Phi(B)-2(\Phi(A)\sharp_{\frac{3}{4}}\Phi(B)))\nonumber\\&
-r_{0}(\Phi(A\sharp B)+\Phi(A)-2\Phi(A\sharp_{\frac{1}{4}} B))\nonumber\\&
\leq2R(\Phi(A)\sharp\Phi(B)-\Phi(A\sharp B)+\frac{1}{2}(\Phi(A)+\Phi(B)-2\Phi(A)\sharp\Phi(B))),
\end{align}}
where $r=\min\{t,1-t\}$, $R=\max\{t,1-t\}$ and $r_{0}=\min\{2r,1-2r\}$.
\end{corollary}
We want to establish some inequalities for positive invertible operators.
\begin{theorem}
Let $A, B\in\mathcal{B}(\mathcal{H})$ be positive invertible. If $t\in [0,1]$ and $\Phi$, $\Psi$ are two unital positive linear maps then for any unit vector $x\in \mathcal{H}$\\
\\
$(\rm i)$ for $0< t\leq\frac{1}{2}$,
\begin{align}\label{th2}
&2r\left(\langle\Phi(A)x,x\rangle\nabla\langle\Psi(B)x,x\rangle-\langle\Phi(A^{\frac{1}{2}})x,x\rangle\langle\Psi(B^{1/2})x,x\rangle\right)\nonumber\\&+
r_{0}\left(\langle\Phi(A^{\frac{1}{2}})x,x\rangle)\langle \Psi(B^{\frac{1}{2}})x,x\rangle+\langle\Phi(A)x,x\rangle-2\langle\Phi(A^{\frac{3}{4}})x,x\rangle \langle\Psi(B^{\frac{1}{4}})x,x\rangle\right)\nonumber\\&
\leq(1-t)\langle\Phi(A)x,x\rangle+t\langle\Psi(B)x,x\rangle-\langle\Psi(B^{t})x,x\rangle\langle\Phi(A^{1-t})x,x\rangle\nonumber\\&
\leq
2R\left(\langle\Phi(A)x,x\rangle\nabla\langle\Psi(B)x,x\rangle-\langle\Phi(A^{\frac{1}{2}})x,x\rangle \langle\Psi(B^{\frac{1}{2}})x,x\rangle\right)\nonumber\\&-r_{0}\left(\langle\Phi(A^{\frac{1}{2}})x,x\rangle \langle\Psi(B^{\frac{1}{2}})x,x\rangle+\langle\Psi(B)x,x\rangle-2\langle\Phi(A^{\frac{1}{4}})x,x\rangle \langle\Psi(B^{\frac{3}{4}})x,x\rangle\right);
\end{align}
$(\rm ii)$ for $\frac{1}{2}<t<1$,
\begin{align*}
&R\left(\langle\Phi(A)x,x\rangle+\langle\Psi(B)x,x\rangle-2\langle\Phi(A^{\frac{1}{2}})x,x\rangle\langle\Psi(B^{1/2})x,x\rangle\right)\\&+
r_{0}\left(\langle\Phi(A^{\frac{1}{2}})x,x\rangle)\langle \Psi(B^{\frac{1}{2}})x,x\rangle+\langle\Phi(A)x,x\rangle-2\langle\Phi(A^{\frac{1}{4}})x,x\rangle \langle\Psi(B^{\frac{3}{4}})x,x\rangle\right)\\&
\leq(1-t)\langle\Phi(A)x,x\rangle+t\langle\Psi(B)x,x\rangle-\langle\Psi(B^{t})x,x\rangle\langle\Phi(A^{1-t})x,x\rangle\\&
\leq
r\left(\langle\Phi(A)x,x\rangle+\langle\Psi(B)x,x\rangle-2\langle\Phi(A^{\frac{1}{2}})x,x\rangle \langle\Psi(B^{\frac{1}{2}})x,x\rangle\right)\\&-r_{0}\left(\langle\Phi(A^{\frac{1}{2}})x,x\rangle \langle\Psi(B^{\frac{1}{2}})x,x\rangle+\langle\Phi(A)x,x\rangle-2\langle\Phi(A^{\frac{3}{4}})x,x\rangle \langle\Psi(B^{\frac{1}{4}})x,x\rangle\right),
\end{align*}
where $r=\min\{t,1-t\}$, $R=\max\{t,1-t\}$, $r_{0}=\min\{2r,1-2r\}$.
\begin{proof}
The proof of part $(\rm ii)$ is similar to the proof of part $(\rm i)$. Thus we only prove $(\rm i)$.\\
Applying inequality \eqref{ww} for any positive real numbers $k, s$, we have
\begin{align}\label{pp1}
r(k+s-2\sqrt{ks})+&r_{0}(k^{\frac{1}{2}}s^{\frac{1}{2}}+k-2k^{\frac{3}{4}}s^{\frac{1}{4}})\nonumber\\&\leq(1-t)k+ts-k^{1-t}s^{t}\nonumber\\&
\leq R(k+s-2\sqrt{ks})-r_{0}(k^{\frac{1}{2}}s^{\frac{1}{2}}+s-2k^{\frac{1}{4}}s^{\frac{3}{4}}).
\end{align}
Fix $s$ in \eqref{pp1}. Then applying functional calculus to the operator  $A$, we have
\begin{align}\label{pp2}
r(A+sI-2\sqrt{s}A^{\frac{1}{2}})+&r_{0}(A^{\frac{1}{2}}s^{\frac{1}{2}}+A-2A^{\frac{3}{4}}s^{\frac{1}{4}})\nonumber\\&
\leq(1-t)A+tsI-s^{t}A^{1-t}\nonumber\\&
\leq R(A+sI-2\sqrt{s}A^{\frac{1}{2}})-r_{0}(A^{\frac{1}{2}}s^{\frac{1}{2}}+sI-2A^{\frac{1}{4}}s^{\frac{3}{4}}).
\end{align}
If we apply the positive linear map $\Phi$  and  inner product for $x\in {\mathcal H}$ with $\|x\|=1$ in inequality \eqref{pp2}, we have
{\footnotesize\begin{align*}
&r\left(\langle\Phi(A)x,x\rangle+s-2\sqrt{s}\langle\Phi(A^{\frac{1}{2}})x,x\rangle\right)\\&+
r_{0}\left(\langle\Phi(A^{\frac{1}{2}})x,x\rangle)s^{\frac{1}{2}}+\langle\Phi(A)x,x\rangle-2\langle\Phi(A^{\frac{3}{4}})x,x\rangle s^{\frac{1}{4}}\right)\\&
\leq(1-t)\langle\Phi(A)x,x\rangle+ts-s^{t}\langle\Phi(A^{1-t})x,x\rangle\nonumber\\&
\leq R\left(\langle\Phi(A)x,x\rangle+s-2\sqrt{s}\langle\Phi(A^{\frac{1}{2}})x,x\rangle\right)-r_{0}\left(\langle\Phi(A^{\frac{1}{2}})x,x\rangle s^{\frac{1}{2}}+s-2\langle\Phi(A^{\frac{1}{4}})x,x\rangle s^{\frac{3}{4}}\right).
\end{align*}}
Now, using the functional calculus to the operator B, we have
{\footnotesize\begin{align*}
&r\left(\langle\Phi(A)x,x\rangle+B-2\langle\Phi(A^{\frac{1}{2}})x,x\rangle B^{1/2}\right)\\&+
r_{0}\left(\langle\Phi(A^{\frac{1}{2}})x,x\rangle)B^{\frac{1}{2}}+\langle\Phi(A)x,x\rangle-2\langle\Phi(A^{\frac{3}{4}})x,x\rangle B^{\frac{1}{4}}\right)\\&
\leq(1-t)\langle\Phi(A)x,x\rangle+tB-B^{t}\langle\Phi(A^{1-t})x,x\rangle\nonumber\\&
\leq R\left(\langle\Phi(A)x,x\rangle+B-2\langle\Phi(A^{\frac{1}{2}})x,x\rangle B^{\frac{1}{2}}\right)-r_{0}\left(\langle\Phi(A^{\frac{1}{2}})x,x\rangle B^{\frac{1}{2}}+B-2\langle\Phi(A^{\frac{1}{4}})x,x\rangle B^{\frac{3}{4}}\right).
\end{align*}}
Taking the positive linear map $\Psi$ and inner product for $y\in {\mathcal H}$ with $\|y\|=1$, we get
\begin{align*}
&r\left(\langle\Phi(A)x,x\rangle+\langle\Psi(B)y,y\rangle-2\langle\Phi(A^{\frac{1}{2}})x,x\rangle\langle\Psi(B^{1/2})x,x\rangle\right)\\&+
r_{0}\left(\langle\Phi(A^{\frac{1}{2}})x,x\rangle)\langle \Psi(B^{\frac{1}{2}})x,x\rangle+\langle\Phi(A)x,x\rangle-2\langle\Phi(A^{\frac{3}{4}})x,x\rangle \langle\Psi(B^{\frac{1}{4}})x,x\rangle\right)\\&
\leq(1-t)\langle\Phi(A)x,x\rangle+t\langle\Psi(B)x,x\rangle-\langle\Psi(B^{t})x,x\rangle\langle\Phi(A^{1-t})x,x\rangle\nonumber\\&
\leq R\left(\langle\Phi(A)x,x\rangle+\langle\Psi(B)x,x\rangle-2\langle\Phi(A^{\frac{1}{2}})x,x\rangle \langle(B^{\frac{1}{2}})x,x\rangle\right)\\&
-r_{0}\left(\langle\Phi(A^{\frac{1}{2}})x,x\rangle \langle\Psi(B^{\frac{1}{2}})x,x\rangle+\langle\Psi(B)x,x\rangle-2\langle\Phi(A^{\frac{1}{4}})x,x\rangle \langle\Psi(B^{\frac{3}{4}})x,x\rangle\right).
\end{align*}
Now, if we put $x=y$, then we get the desired result.
\end{proof}
\end{theorem}
\begin{theorem}\label{waw}
Let $A, B\in\mathcal{B}(\mathcal{H})$ be positive   invertible. If $t\in [0,1]$ and $\Phi$, $\Psi$ are two unital positive linear maps, then for any unit vector $x\in \mathcal{H}$\\
$(\rm i)$ for $0< t\leq\frac{1}{2}$,
{\footnotesize\begin{align*}
&2r\left(\langle\Phi(A)x,x\rangle\nabla\langle\Psi(B)x,x\rangle-\langle\Phi(A^{\frac{1}{2}})x,x\rangle\langle\Psi(B)x,x\rangle^{1/2}\right)\nonumber\\&+
r_{0}\left(\langle\Phi(A^{1/2})x,x\rangle\langle \Psi(B)x,x\rangle^{1/2}+\langle\Phi(A)x,x\rangle-2\langle\Phi(A^{3/4})x,x\rangle \langle\Psi(B)x,x\rangle^{1/4}\right)\nonumber\\&
\leq(1-t)\langle\Phi(A)x,x\rangle+t\langle\Psi(B)x,x\rangle-\langle\Psi(B)x,x\rangle^{t}\langle\Phi(A^{1-t})x,x\rangle\nonumber\\&
\leq
2R\left(\langle\Phi(A)x,x\rangle\nabla\langle\Psi(B)x,x\rangle-\langle\Phi(A^{1/2})x,x\rangle \langle\Psi(B)x,x\rangle^{1/2}\right)\nonumber\\&-r_{0}\left(\langle\Phi(A^{1/2})x,x\rangle \langle\Psi(B)x,x\rangle^{1/2}+\langle\Psi(B)x,x\rangle-2\langle\Phi(A^{1/4})x,x\rangle \langle\Psi(B)x,x\rangle^{3/4}\right);
\end{align*}}
$(\rm ii)$ for $\frac{1}{2}< t<1$,
{\footnotesize\begin{align*}
&R\left(\langle\Phi(A)x,x\rangle+\langle\Psi(B)x,x\rangle-2\langle\Phi(A^{1/2})x,x\rangle\langle\Psi(B)x,x\rangle^{1/2}\right)\\&+
r_{0}\left(\langle\Phi(A^{1/2})x,x\rangle)\langle \Psi(B)x,x\rangle^{1/2}+\langle\Phi(A)x,x\rangle-2\langle\Phi(A^{1/4})x,x\rangle \langle\Psi(B)x,x\rangle^{3/4}\right)\\&
\leq(1-t)\langle\Phi(A)x,x\rangle+t\langle\Psi(B)x,x\rangle-\langle\Psi(B)x,x\rangle^{t}\langle\Phi(A^{1-t})x,x\rangle\\&
\leq
r\left(\langle\Phi(A)x,x\rangle+\langle\Psi(B)x,x\rangle-2\langle\Phi(A^{1/2})x,x\rangle \langle\Psi(B)x,x\rangle^{1/2}\right)\\&-r_{0}\left(\langle\Phi(A^{1/2})x,x\rangle \langle\Psi(B)x,x\rangle^{1/2}+\langle\Phi(A)x,x\rangle-2\langle\Phi(A^{3/4})x,x\rangle \langle\Psi(B)x,x\rangle^{1/4}\right),
\end{align*}}
where $r=\min\{t,1-t\}$, $R=\max\{t,1-t\}$, $r_{0}=\min\{2r,1-2r\}$.
\end{theorem}
\begin{proof}
The proof of part $(\rm ii)$ is similar to the proof of part $(\rm i)$. Thus we just prove $(\rm i)$.
For any positive real number $k$ and any unit vector $x\in \mathcal H$, we have
{\footnotesize\begin{align}\label{ppp1}
&r\left(k+\langle\Psi(B)x,x\rangle-2\sqrt{k}\langle\Psi(B)x,x\rangle\right)+r_{0}\left(k^{1/2}\langle\Psi(B)x,x\rangle^{1/2}
+k-2k^{3/4}\langle\Psi(B)x,x\rangle^{1/4}\right)\nonumber\\&\leq(1-t)k+t\langle\Psi(B)x,x\rangle-k^{1-t}\langle\Psi(B)x,x\rangle^{t}\nonumber\\&
\leq R\left(k+\langle\Psi(B)x,x\rangle-2\sqrt{k}\langle\Psi(B)x,x\rangle^{1/2}\right)\nonumber\\&\,\,\,-r_{0}\left(k^{1/2}\langle\Psi(B)x,x\rangle^{1/2}+
\langle\Psi(B)x,x\rangle-2k^{1/4}\langle\Psi(B)x,x\rangle^{3/4}\right).
\end{align}}
Applying inequality \eqref{ppp1} and the functional calculus for the operator $A$, we have
{\footnotesize\begin{align}\label{ppp2}
r\big(A+\langle\Psi(B)x,x\rangle I_{\mathcal H}&-2\sqrt{A}\langle\Psi(B)x,x\rangle\big)\\&+r_{0}\big(A^{1/2}\langle\Psi(B)x,x\rangle^{1/2}
+A-2A^{3/4}\langle\Psi(B)x,x\rangle^{1/4}\big)\nonumber\\&\leq(1-t)A+t\langle\Psi(B)x,x\rangle I_{\mathcal H}-A^{1-t}\langle\Psi(B)x,x\rangle^{t}\nonumber\\&
\leq R\big(A+\langle\Psi(B)x,x\rangle I_{\mathcal H}-2\sqrt{A}\langle\Psi(B)x,x\rangle^{1/2}\big)\nonumber\\&-r_{0}\big(A^{1/2}\langle\Psi(B)x,x\rangle^{1/2}+
\langle\Psi(B)x,x\rangle I_{\mathcal H}-2A^{1/4}\langle\Psi(B)x,x\rangle^{3/4}\big).
\end{align}}
Now, using the unital positive operator $\Phi$  and  the inner product for $y\in \mathcal H$ with $\|y\|=1$ in inequality \eqref{ppp2}, we get
{\footnotesize\begin{align*}
&r\left(\langle\Phi(A)y,y\rangle+\langle\Psi(B)x,x\rangle I_{\mathcal H}-2\langle\Phi(A)y,y\rangle^{1/2}\langle\Psi(B)x,x\rangle\right)\\&+r_{0}\left(\langle\Phi(A^{1/2})y,y\rangle\langle\Psi(B)x,x\rangle^{1/2}
+\langle\Phi(A)y,y\rangle-2\langle\Phi(A^{3/4})y,y\rangle\langle\Psi(B)x,x\rangle^{1/4}\right)\\&\leq(1-t)\langle\Phi(A)y,y\rangle+t\langle\Psi(B)x,x\rangle I_{\mathcal H}-\langle\Phi(A^{1-t})y,y\rangle\langle\Psi(B)x,x\rangle^{t}\nonumber\\&
\leq R\left(\langle\Phi(A)y,y\rangle+\langle\Psi(B)x,x\rangle I_{\mathcal H}-2\langle\Phi(A)y,y\rangle^{1/2}\langle\Psi(B)x,x\rangle^{1/2}\right)\\&-r_{0}\left(\langle\Phi(A^{1/2})y,y\rangle\langle\Psi(B)x,x\rangle^{1/2}+
\langle\Psi(B)x,x\rangle I_{\mathcal H}-2\langle\Phi(A^{1/4})y,y\rangle\langle\Psi(B)x,x\rangle^{3/4}\right).
\end{align*}}
Now,  putting $y=x$, we get the desired result.
\end{proof}
\begin{corollary}\label{cor}
Let $A\in\mathcal{B}(\mathcal{H})$ be positive,  $\Phi$ is a unital positive linear map and  $t\in [0,1]$. Then for any unit vector $x\in \mathcal{H}$\\
$(\rm i)$ for $0< t\leq\frac{1}{2}$,
{\footnotesize\begin{align*}
&2r\langle \Phi(A)x,x\rangle^{t-\frac{1}{2}}\left(\langle\Phi(A)x,x\rangle^{\frac{1}{2}}-\langle\Phi(A^{\frac{1}{2}})x,x\rangle\right)\\&+
r_{0}\langle\Phi(A)x,x\rangle^{t-\frac{1}{2}}\left(\langle\Phi(A^{1/2})x,x\rangle+
\langle\Phi(A)x,x\rangle^{1/2}-2\langle\Phi(A^{3/4})x,x\rangle\langle\Phi(A)x,x\rangle^{-1/4}\right)\\&
\leq\langle\Phi(A)x,x\rangle^{t}-\langle\Phi(A^t)x,x\rangle\\&
\leq2R\langle \Phi(A)x,x\rangle^{t-\frac{1}{2}}\left(\langle\Phi(A)x,x\rangle^{\frac{1}{2}}-\langle\Phi(A^{\frac{1}{2}})x,x\rangle\right)\\&-
r_{0}\langle\Phi(A)x,x\rangle^{t-\frac{1}{2}}\left(\langle\Phi(A^{1/2})x,x\rangle+
\langle\Phi(A)x,x\rangle^{1/2}-2\langle\Phi(A^{1/4})x,x\rangle\langle\Phi(A)x,x\rangle^{1/4}\right);
\end{align*}}
$(\rm ii)$ for $\frac{1}{2}< t<1$,
{\footnotesize\begin{align*}
&2R\langle \Phi(A)x,x\rangle^{t-\frac{1}{2}}\left(\langle\Phi(A)x,x\rangle^{\frac{1}{2}}-\langle\Phi(A^{\frac{1}{2}})x,x\rangle\right)\\&+
r_{0}\langle\Phi(A)x,x\rangle^{t-\frac{1}{2}}\left(\langle\Phi(A^{1/2})x,x\rangle+
\langle\Phi(A)x,x\rangle^{1/2}-2\langle\Phi(A^{1/4})x,x\rangle\langle\Phi(A)x,x\rangle^{1/4}\right)\\&
\leq\langle\Phi(A)x,x\rangle^{t}-\langle\Phi(A^t)x,x\rangle\\&
\leq2r\langle \Phi(A)x,x\rangle^{t-\frac{1}{2}}\left(\langle\Phi(A)x,x\rangle^{\frac{1}{2}}-\langle\Phi(A^{\frac{1}{2}})x,x\rangle\right)\\&-
r_{0}\langle\Phi(A)x,x\rangle^{t-\frac{1}{2}}\left(\langle\Phi(A^{1/2})x,x\rangle+
\langle\Phi(A)x,x\rangle^{1/2}-2\langle\Phi(A^{3/4})x,x\rangle\langle\Phi(A)x,x\rangle^{-1/4}\right),
\end{align*}}
where $r=\min\{t,1-t\}$, $R=\max\{t,1-t\}$, $r_{0}=\min\{2r,1-2r\}$.
\end{corollary}
\begin{proof}
Letting $\Psi=\Phi$ and $B=A$ in Theorem \ref{waw}, we get the desired inequalities.
\end{proof}
In the next result we obtain a refinement of inequality \eqref{mc1} for $t\in(0,\frac{1}{2}]$.
\begin{corollary}
Let $T\in {\mathcal B}(\mathcal H)$ be positive operator and  $x\in {\mathcal H}$ be a unit vector. Then for $0< t\leq\frac{1}{2}$, we have
{\footnotesize\begin{align*}
\langle Tx,x&\rangle^{t}-\langle T^{t}x,x\rangle\\&\leq2R\langle Tx,x\rangle^{t-\frac{1}{2}}\left(\langle Tx,x\rangle^\frac{1}{2}-\langle T^\frac{1}{2}x,x\rangle\right)\\&-
r_{0}\langle Tx,x\rangle^{t-\frac{1}{2}}\left(\langle T^\frac{1}{2}x,x\rangle+
\langle Tx,x\rangle^\frac{1}{2}-2\langle T^\frac{1}{4}x,x\rangle\langle Tx,x\rangle^\frac{1}{4}\right)\\&
\leq2R\left(\langle Tx,x\rangle^\frac{1}{2}-\langle T^\frac{1}{2}x,x\rangle\right)-r_{0}\left(\langle T^\frac{1}{2}x,x\rangle+
\langle Tx,x\rangle^\frac{1}{2}-2\langle T^\frac{1}{4}x,x\rangle\langle Tx,x\rangle^\frac{1}{4}\right),
\end{align*}}
where $r=\min\{t,1-t\}$, $R=\max\{t,1-t\}$, $r_{0}=\min\{2r,1-2r\}$.
\end{corollary}
\begin{proof}
If we replace $\Phi(A)=A, A\in\mathcal{B}(\mathcal{H})$ and $t$ with $1-t$ in Corollary \ref{cor}, then we get the desired result.
\end{proof}
\section{Applications}
In this section, we prove  difference reverse-types of the  H\"{o}lder and  Cauchy-Schwarz inequalities.
\begin{theorem}\label{TH}
 Let $A_i, B_i\in\mathcal{B}(\mathcal{H})\,\,(1\leq i \leq n)$ be positive  invertible  and  $t\in[0,1]$.\\
$(\rm i)$ If $0< t\leq\frac{1}{2}$, then
{\footnotesize\begin{align*}
&\left(\sum_{i=1}^{n}A_{i}\right)\sharp_{t}\left(\sum_{i=1}^{n}B_{i}\right)-\left(\sum_{i=1}^{n}A_{i}\sharp_{t}B_{i}\right)\\&\leq R\left(\sum_{i=1}^{n}A_{i}+\sum_{i=1}^{n}B_{i}-2\sum_{i=1}^{n}(A_{i}\sharp B_{i})\right)-r_{0}\left(\sum_{i=1}^{n}(A_{i}\sharp B_{i})+\sum_{i=1}^{n}B_{i}-2\sum_{i=1}^{n}(A_{i}\sharp_{\frac{3}{4}}B_{i})\right)
\\&-r_{0}\left(\sum_{i=1}^{n}A_{i}\sharp\sum_{i=1}^{n}B_{i}+\sum_{i=1}^{n}A_{i}-2\left(\sum_{i=1}^{n}A_{i}\sharp_{\frac{1}{4}}\sum_{i=1}^{n}B_{i}\right)\right).
\end{align*}}
$(\rm ii)$ If $\frac{1}{2}<t< 1$, then
{\footnotesize\begin{align*}
\left(\sum_{i=1}^{n}A_{i}\right)&\sharp_{t}\left(\sum_{i=1}^{n}B_{i}\right)-\Bigg(\sum_{i=1}^{n}A_{i}\sharp_{t}B_{i}\Bigg)
\leq R\left(\sum_{i=1}^{n}A_{i}+\sum_{i=1}^{n}B_{i}-2\left(\sum_{i=1}^{n}A_{i}\sharp B_{i}\right)\right)\\&-r_{0}\left((\sum_{i=1}^{n}A_{i})\sharp (\sum_{i=1}^{n}B_{i})+\sum_{i=1}^{n}B_{i}-2\left((\sum_{i=1}^{n}A_{i})\sharp_{\frac{3}{4}}(\sum_{i=1}^{n}B_{i})\right)\right)\\&
-r_{0}\left(\sum_{i=1}^{n}(A_{i}\sharp B_{i})+\sum_{i=1}^{n}A_{i}-2\sum_{i=1}^{n}(A_{i}\sharp_{\frac{1}{4}} B_{i})\right),
\end{align*}}
where $r=\min\{t,1-t\}$, $R=\max\{t,1-t\}$ and $r_{0}=\min\{2r,1-2r\}$.
\end{theorem}
\begin{proof}
Taking $A={\rm diag}(A_{1},\cdots,A_{n})$, $B={\rm diag}(B_{1},\cdots,B_{n})$ and $\Phi\left([C_{ij}]_{1\leq i,j\leq n}\right)=\sum_{i=1}^{n}C_{ii}$ in equalities \eqref{i1} and \eqref{i2}, we get the desired inequality.
\end{proof}
Since the function $f(x)=x^t\,\,(t\in[0,1])$ is an operator concave function,
$\sum_{i=1}^{n}w_{i}T_{i}^{t}\leq\Big(\sum_{i=1}^{n}w_{i}T_{i}\Big)^{t}$ for positive  operators $T_i$ and  positive real numbers $w_{i}$ such that $\sum_{i=1}^{n}w_{i}=1$. Now,  Theorem \ref{TH} yields a reverse of this inequality as follows:
\begin{example}
If for positive  operators $T_i\,\,(1\leq i \leq n)$, we take $A_i=w_{i}I$ and $B_i=w_{i}T_i\,\,(1\leq i \leq n)$, in Theorem \ref{TH}, where $w_{i}$'s are positive real numbers such that $\sum_{i=1}^{n}w_{i}=1$, we reach the following inequalities:\\
$(\rm i)$ If $0\leq t\leq\frac{1}{2}$, then
\begin{align*}
\Big(\sum_{i=1}^{n}w_{i}T_{i}\Big)^{t}-\sum_{i=1}^{n}w_{i}T_{i}^{t}\leq &R\Big(I+\sum_{i=1}^{n}w_{i}T_{i}-2\sum_{i=1}^{n}w_{i}T_{i}^{1/2}\Big)\\&-r_{0}\Big(\sum_{i=1}^{n}w_{i}T_{i}^{1/2}+
\sum_{i=1}^{n}w_{i}T_{i}-2\sum_{i=1}^{n}w_{i}T_{i}^{3/4}\Big)\\&-r_{0}\Big((\sum_{i=1}^{n}w_{i}T_{i})^{1/2}+I-2(\sum_{i=1}^{n}w_{i}T_{i})^{1/4}\Big).
\end{align*}
$(\rm ii)$ If $\frac{1}{2}<t\leq1$, then
\begin{align*}
\Big(\sum_{i=1}^{n}w_{i}T_{i}\Big)^{t}-\sum_{i=1}^{n}w_{i}T_{i}^{t}\leq &R\Big(I+\sum_{i=1}^{n}w_{i}T_{i}-2\sum_{i=1}^{n}w_{i}T_{i}^{1/2}\Big)\\&-r_{0}\Big((\sum_{i=1}^{n}w_{i}T_{i})^{1/2}+
\sum_{i=1}^{n}w_{i}T_{i}-2(\sum_{i=1}^{n}w_{i}T_{i})^{3/4}\Big)\\&-r_{0}\Big(\sum_{i=1}^{n}w_{i}T_{i}^{1/2}+I-2\sum_{i=1}^{n}w_{i}T_{i}^{1/4}\Big).
\end{align*}
\end{example}
In \cite{ya}, the Tsallis relative operator entropy $T_{t}(A|B)$ for positive invertible operators $A$, $B$ and $0<t\leq1$ is defined as follows:
\begin{align*}
T_{t}(A,B)=\frac{A\sharp_{t}B-A}{t}.
\end{align*}
For further information about  the Tsallis relative operator entropy see \cite{fuj3} and references therein. In \cite[Proposition 2.3]{Fur}, it is shown that for any unital positive linear map $\Phi$ the following inequality  holds:
\begin{align}\label{tsa1}
\Phi(T_{t}(A|B))\leq T_{t}(\Phi(A)|\Phi(B)).
\end{align}
 In \eqref{tsa1}, by similar techniques of Theorem \ref{TH}, for  positive operators $A_{i}, B_{i}$ $(i=1,2,...,n)$, we have
\begin{align}\label{tsa}
\sum_{i=1}^{n}\left(T_{t}(A_{i}|B_{i})\right)\leq T_{t}\left(\sum_{i=1}^{n}A_{i}|\sum_{i=1}^{n}B_{i}\right).
\end{align}
In the next theorem we show a reverse of  inequality \eqref{tsa}.
\begin{theorem}
Let $A_i, B_i\in\mathcal{B}(\mathcal{H})\,\,(1\leq i\leq n)$ be positive invertible  and $t\in (0,1)$. \\
$(\rm i)$ If $0<t\leq\frac{1}{2}$, then
{\footnotesize\begin{align*}
&T_{t}\left(\sum_{i=1}^{n}A_{i}|\sum_{i=1}^{n}B_{i}\right)-\sum_{i=1}^{n}(T_{t}(A_{i}|B_{i}))\\&
\leq\frac{1}{t}\Bigg[R\Big(\sum_{i=1}^{n}A_{i}+\sum_{i=1}^{n}B_{i}-2\sum_{i=1}^{n}(A_{i}\sharp B_{i})\Big)-r_{0}\Big(\sum_{i=1}^{n}(A_{i}\sharp B_{i})+\sum_{i=1}^{n}B_{i}-2\sum_{i=1}^{n}(A_{i}\sharp_{\frac{3}{4}}B_{i})\Big)
\\&-r_{0}\Big(\sum_{i=1}^{n}A_{i}\sharp\sum_{i=1}^{n}B_{i}+\sum_{i=1}^{n}A_{i}-2(\sum_{i=1}^{n}A_{i}\sharp_{\frac{1}{4}}\sum_{i=1}^{n}B_{i})\Big)\Bigg].
\end{align*}}
$(\rm ii)$ If $\frac{1}{2}<t<1$, then
{\footnotesize\begin{align*}
T_{t}\left(\sum_{i=1}^{n}A_{i}|\sum_{i=1}^{n}B_{i}\right)&-\sum_{i=1}^{n}(T_{t}(A_{i}|B_{i}))\leq\frac{1}{t}\Bigg[R\Big(\sum_{i=1}^{n}A_{i}+\sum_{i=1}^{n}B_{i}-2(\sum_{i=1}^{n}A_{i}\sharp B_{i})\Big)\\&-r_{0}\Big((\sum_{i=1}^{n}A_{i})\sharp (\sum_{i=1}^{n}B_{i})+\sum_{i=1}^{n}B_{i}-2(\sum_{i=1}^{n}A_{i}\sharp_{\frac{3}{4}}\sum_{i=1}^{n}B_{i})\Big)\\&
-r_{0}\Big(\sum_{i=1}^{n}(A_{i}\sharp B_{i})+\sum_{i=1}^{n}A_{i}-2\sum_{i=1}^{n}(A_{i}\sharp_{\frac{1}{4}} B_{i})\Big)\Bigg].
\end{align*}}
\end{theorem}
\begin{proof}
 Applying Theorem \ref{TH} for $0<t\leq\frac{1}{2}$, we have
{\footnotesize\begin{align*}
&T_{t}\left(\sum_{i=1}^{n}A_{i}|\sum_{i=1}^{n}B_{i}\right)-\sum_{i=1}^{n}(T_{t}(A_{i}|B_{i}))\\&=\frac{\left(\sum_{i=1}^{n}A_{i}\right)\sharp_{t}\left(\sum_{i=1}^{n}B_{i}\right)-\sum_{i=1}^{n}A_{i}}{t}-\sum_{i=1}^{n}\frac{A_{i}\sharp_{t}B_{i}-A_{i}}{t}
\\&\leq\frac{1}{t}\Bigg[R\Big(\sum_{i=1}^{n}A_{i}+\sum_{i=1}^{n}B_{i}-2\sum_{i=1}^{n}(A_{i}\sharp B_{i})\Big)-r_{0}\Big(\sum_{i=1}^{n}(A_{i}\sharp B_{i})+\sum_{i=1}^{n}B_{i}-2\sum_{i=1}^{n}(A_{i}\sharp_{\frac{3}{4}}B_{i})\Big)
\\&-r_{0}\Big(\sum_{i=1}^{n}A_{i}\sharp\sum_{i=1}^{n}B_{i}+\sum_{i=1}^{n}A_{i}-2(\sum_{i=1}^{n}A_{i}\sharp_{\frac{1}{4}}\sum_{i=1}^{n}B_{i})\Big)\Bigg],
\end{align*}}
whence we get the first inequality. The proof of the second inequality is similar.
\end{proof}
\begin{remark}
We can express our results for non-invertible operators; see \cite{fuj3}.
It is a direct consequence of the definition of the mean in the sense of Kubo-Ando \cite{mm} that
$A\sharp_t (B+\varepsilon)$ is a monotone increasing net.
Let $B$ be a non-invertible operator and $ \epsilon>0$.  It follows from the set $\{A\sharp_{t}(B+\epsilon): \epsilon>0\}$ is bounded above for $0 < \varepsilon <1$ that the limit \begin{align}\label{inv}
A\sharp_{t}B=\lim_{\epsilon\downarrow0}A\sharp_{t}(B+\epsilon)
\end{align}
exists in the strong operator topology. So by \eqref{inv}, $A\sharp_{t}B$ exists.\\
\end{remark}

\section*{Competing interests}
  The authors declare that they have no competing interests.
\section*{Author's contributions}
   The authors contributed equally to the manuscript and read and approved the final manuscript.
\section*{Acknowledgement.}
The  third  author   would  like  to  thank  the Tusi Mathematical Research Group (TMRG).

\bibliographystyle{amsplain}

\end{document}